\documentclass{article}
\usepackage{amsmath}
\usepackage{amsfonts}
\usepackage{latexsym}
\usepackage{amsthm}
\usepackage{mathrsfs}
\usepackage{amscd}
\usepackage{amsmath}
\usepackage{amssymb}
\usepackage[all, warning]{onlyamsmath}
\usepackage[all]{xy}

\newcommand{\Hom}{\mathop{\mathrm{Hom}}\nolimits}

\newcommand{\Coker}{\mathop{\mathrm{Coker}}\nolimits}

\newcommand{\Spec}{\mathop{\mathrm{Spec}}\nolimits}

\newcommand{\Mod}{\mathop{\mathcal{M}od_k}\nolimits}
\newcommand{\ModA}{\mathop{\mathcal{M}od_k^{\mathbb{A}^1}}\nolimits}
\newcommand{\Ab}{\mathop{\mathcal{A}b_k}\nolimits}
\newcommand{\AbA}{\mathop{\mathcal{A}b_k^{\mathbb{A}^1}}\nolimits}
\newcommand{\Smk}{\mathop{\mathcal{S}m_k}\nolimits}
\newcommand{\Spck}{\mathop{\mathcal{S}pc_k}\nolimits}
\newcommand{\aone}{\mathop{\mathbb{A}^1}\nolimits}

\newtheorem{thm}{Theorem}[section]
\newtheorem{lem}[thm]{Lemma}
\newtheorem{prop}[thm]{Proposition}
\newtheorem{cor}[thm]{Corollary}
\theoremstyle{definition}
\newtheorem{defi}[thm]{Definition}

\newtheorem{rem}[thm]{Remark}

\newtheorem*{nota}{Notation}
\newtheorem*{acknow}{Acknowledgments}

\makeatletter

\@addtoreset{equation}{section}
\makeatother

\begin{document}

\title{Relative $\aone$-homology and its applications}
\author{Yuri Shimizu \footnote{Department of Mathematics, Tokyo Institute of Technology \newline E-mail: qingshuiyouli27@gmail.com}}
\maketitle

\begin{abstract}
In this paper, we prove an $\aone$-homology version of the Whitehead theorem with dimension bound. We also prove an excision theorem for $\aone$-homology, Suslin homology and $\aone$-homotopy sheaves. In order to prove these results, we develop a general theory of relative $\aone$-homology and $\aone$-homotopy sheaves. As an application, we compute the relative  $\aone$-homology of a hyperplane embedding $\mathbb{P}^{n-1} \hookrightarrow \mathbb{P}^{n}$.
\end{abstract}

\section{Introduction}

In this paper, we study a relative version of the $\aone$-homology sheaves of smooth schemes over a field and give applications to motives and $\aone$-homotopy theory. Let $k$ be a field. In \cite{Vo00}, Voevodsky constructed a triangulated category of motives $\mathbf{DM}_{-}^{eff}(k)$ over $k$ as a triangulated subcategory of the derived category of Nisnevich sheaves with transfers, with a functor $M$ from the category of smooth $k$-schemes $\Smk$ to $\mathbf{DM}_{-}^{eff}(k)$. For $X \in \Smk$, $M(X)$ is called the motive of $X$. Its homology sheaves with transfers $\mathbf{H}^S_*(X) = H_*(M(X))$ are called the Suslin homology sheaves (\textit{cf}. \cite[Section 2]{As}), whose sections over $\Spec k$ give the Suslin homology group introduced by Suslin-Voevodsky \cite{SV} when $k$ is perfect.

In \cite{MV}, Morel-Voevodsky established the $\aone$-homotopy theory and defined an $\aone$-version of homotopy groups, called $\aone$-homotopy sheaves, as Nisnevich sheaves on $\Smk$. Morel \cite{Mo2} introduced an $\aone$-version of the singular homology, called $\aone$-homology sheaves, as an analogue of Suslin homology by instead using Nisnevich sheaves without transfers. As with motives, there is a functor $C^{\aone}$ from $\Smk$ to a triangulated subcategory of the derived category of Nisnevich sheaves on $\Smk$. The $\aone$-homology sheaves $\mathbf{H}^{\aone}_*(X)$ of $X \in \Smk$ are defined as the homology sheaves $H_*(C^{\aone}(X))$. 

The purpose of this paper is threefold. Firstly, we prove an $\aone$-homological Whitehead theorem with dimension bound and the excision theorem. Secondly, as a tool for proving them, we develop a general theory of \textit{relative} $\aone$-homology, namely $\aone$-homology sheaves for morphisms $f : X \to Y$. Thirdly, as an example, we compute the relative $\aone$-homology of a hyperplane embedding $\mathbb{P}^{n-1} \hookrightarrow \mathbb{P}^{n}$.

Our $\aone$-Whitehead theorem detects whether a morphism $X \to Y$ is an $\aone$-weak equivalence in terms of the $\aone$-fundamental groups and the $\aone$-homology up to degree ${\max\{ \dim X + 1, \dim Y\}}$. See \cite[Thm. 6.4.15]{Ar} for the classical homological Whitehead theorem in topology.

\begin{thm}[see Theorem \ref{dim cor A1}]\label{intro dim cor A1}
Assume $k$ perfect. Let $f : (X,*) \to (Y,*)$ be a morphism of $\aone$-simply connected pointed smooth $k$-schemes and let $d = {\max\{ \dim X + 1, \dim Y\}}$. If $f$ induces an isomorphism $\mathbf{H}^{\aone}_i(X) \xrightarrow{\cong} \mathbf{H}^{\aone}_i(Y)$ for all $2 \leq i < d$ and an epimorphism $\mathbf{H}^{\aone}_d(X) \twoheadrightarrow \mathbf{H}^{\aone}_d(Y)$, then $f$ is an $\aone$-weak equivalence.
\end{thm}

The Whitehead theorem for $\aone$-\textit{homotopy} sheaves is established by Morel-Voevodsky \cite{MV}, and the novelty here is the detection by $\aone$-\textit{homology} sheaves and the degree bound $d = {\max\{ \dim X + 1, \dim Y\}}$. Next, our excision theorem for $\aone$-homology sheaves is stated as follows.

\begin{thm}[see Theorem \ref{thm. A1-excision}]\label{intro thm ex}
Let $X$ be a smooth $k$-scheme and $U$ a Zariski open set of $X$ whose complement has codimension $r$. Then the morphisms
\begin{align*}
&\mathbf{H}^{\aone}_i(U) \to \mathbf{H}^{\aone}_i(X)& &\mbox{and}& &\mathbf{H}^S_i(U) \to \mathbf{H}^S_i(X)&
\end{align*}
are isomorphisms for every $i < r - 1$ and epimorphisms for $i = r - 1$.
\end{thm}

Asok \cite[Prop. 3.8]{As} proved a similar result in degree $0$. We also obtain the $\aone$-homotopy version of the excision theorem.

\begin{cor}[see Corollary \ref{A1 excision of homotopy}]\label{intro A1 excision of homotopy}
Assume $k$ perfect. Let $(X,*)$ be a pointed smooth $k$-scheme and $(U,*)$ a pointed Zariski open set of $(X,*)$ whose complement has codimension $r$. If $(X,*)$ and $(U,*)$ are $\aone$-simply connected, then the morphism
\begin{equation*}
\pi^{\aone}_i(U,*) \to \pi^{\aone}_i(X,*)
\end{equation*}
is an isomorphism for every $i < r - 1$ and an epimorphism for $i = r - 1$.
\end{cor}

This is a variation of the $\aone$-excision theorem of Asok-Doran \cite{AD} which assumes that $\pi^{\aone}_i(X,*) = 0$ for all $i < r - 1$ and $k$ infinite. In order to prove these results, we develop a general theory of relative $\aone$-homotopy and $\aone$-homology sheaves. If $f : X \to Y$ is a morphism of smooth $k$-schemes, we define its $\aone$-homotopy $\pi^{\aone}_i(f)$, $\aone$-homology $\mathbf{H}^{\aone}_i(f)$ and Suslin homology $\mathbf{H}^{S}_i(f)$.

Finally, as an application of the results above, we compute the relative $\aone$-homology sheaves of the pair $(\mathbb{P}^n,\mathbb{P}^{n-1})$. Let $\underline{\mathbf{K}}^{MW}_n$ be the unramified Milnor-Witt $K$-theory defined by Morel \cite{Mo2}.

\begin{thm}[see Theorem \ref{(Pn,Pn-1)}]\label{intro thm P}
Assume $k$ perfect. For $0 \leq i \leq n$, $n > 0$, we have
\begin{equation*}
\mathbf{H}^{\aone}_i(\mathbb{P}^n,\mathbb{P}^{n-1}) \cong 
\begin{cases}
\underline{\mathbf{K}}^{MW}_n &(i = n) \\
0 &(i < n).
\end{cases}
\end{equation*}
In particular, when $i < n$, we have $\mathbf{H}^{\aone}_i(\mathbb{P}^n) \cong \mathbf{H}^{\aone}_i(\mathbb{P}^{i+1})$.
\end{thm}

Similar stabilization $\mathbf{H}^{S}_i(\mathbb{P}^n) \cong \mathbf{H}^{S}_i(\mathbb{P}^{i+1})$ in $i < n$ also holds for the Suslin homology (Corollary \ref{cor (Pn,Pn-1) sus}). The $\aone$-homology $\mathbf{H}^{\aone}_i(\mathbb{P}^{i+1})$ can be described in terms of $\underline{\mathbf{K}}^{MW}_i$ (Corollary \ref{cor (Pn,Pn-1)}).

This paper is organized as follows. In Section \ref{Relative Hurewicz Whitehead}, we prove a weak version of the relative $\aone$-Hurewicz theorem. In Section \ref{hom eq ex}, we prove Theorems \ref{intro dim cor A1}-\ref{intro thm ex}. In Section \ref{sec. P1}, we prove Theorem \ref{intro thm P}.

\begin{nota}
In this paper, we fix a field $k$. We denote by $\Smk$ the category of smooth $k$-schemes. Every sheaf is considered on the Nisnevich topology on $\Smk$. Objects of an abelian category are regarded as complexes concentrated in degree zero.
\end{nota}

\begin{acknow}
I would like to thank my adviser Shohei Ma for many useful advices. I would also like to thank Tom Bachmann for pointing out a mistake in the previous version. I would like to thank the referee for many detailed comments
which helped us to improve the presentation.
\end{acknow}

\section{Relative $\aone$-homotopy and $\aone$-homology}\label{Relative Hurewicz Whitehead}

In this section, we give basic definitions of relative $\aone$-homotopy and $\aone$-homology, and establish a weak relative $\aone$-Hurewicz theorem. We also compare $\aone$-homology and Suslin homology. We refer to \cite{MV}, \cite{MVW}, \cite{SV}, \cite{Mo2} and \cite{As} for the basic theory of $\aone$-homology and $\aone$-homotopy.

\subsection{Basic definitions}

Let $\Spck$ be the category of simplicial Nisnevich sheaves on $\Smk$ (called $k$-spaces) equipped with the $\aone$-model structure of \cite{MV}. We denote by $\pi_0^{\aone}(\mathcal{X})$ the sheaf of $\aone$-connected components of a $k$-space $\mathcal{X}$ and denote by $\pi_n^{\aone}(\mathcal{X},*)$ the $n$-th $\aone$-homotopy sheaf of a pointed $k$-space $(\mathcal{X},*)$ for $n \geq 0$. A $k$-space $\mathcal{X}$ is called $\aone$-connected if $\pi_0^{\aone}(\mathcal{X}) \cong \Spec k$, and a pointed $k$-space $(\mathcal{X},*)$ is called $\aone$-$n$-connected if $\mathcal{X}$ is $\aone$-connected and if $\pi_i^{\aone}(\mathcal{X},*) = 0$ for all $1 \leq i \leq n$. Especially, $(\mathcal{X},*)$ is called $\aone$-simply connected if it is $\aone$-$1$-connected. We consider a relative version of these definitions.

\begin{defi}\label{def of rel A1 homotopy}
For a morphism of pointed $k$-spaces $f : (\mathcal{X},*) \to (\mathcal{Y},*)$, the $i$-th $\aone$-homotopy sheaf of $f$ is defined by
\begin{equation*}
\pi^{\aone}_i(f) = 
\begin{cases}
\pi^{\aone}_{i-1}(F_{\aone}(f)) &(i > 0) \\
\Spec k &(i = 0),
\end{cases}
\end{equation*}
where $F_{\aone}(f)$ is the homotopy fiber with respect to the $\aone$-model structure of $\Spck$. When $f$ is an inclusion, we write $\pi^{\aone}_i(\mathcal{Y},\mathcal{X}) = \pi^{\aone}_i(f)$. A morphism (or a pair) is called $\aone$-$n$-connected if its $\aone$-homotopy sheaves in degree $\leq n$ are isomorphic to $\Spec k$.
\end{defi}

Since $F_{\aone}(f) \to \mathcal{X} \xrightarrow{f} \mathcal{Y}$ is a homotopy fiber sequence under the $\aone$-model structure, we obtain by \cite[Prop. 4.21]{AE} the long exact sequence
\begin{multline*}
\cdots \to \pi^{\aone}_i(\mathcal{X},*) \to \pi^{\aone}_i(\mathcal{Y},*) \to \pi^{\aone}_i(f) \to \pi^{\aone}_{i-1}(\mathcal{X},*) \to \cdots \\
\cdots \to \pi^{\aone}_0(\mathcal{X},*) \to \pi^{\aone}_0(\mathcal{Y},*) \to 0.
\end{multline*}

We fix a commutative unital ring $R$. Let $\Mod(R)$ be the category of Nisnevich sheaves of $R$-modules on $\Smk$ and $\mathbf{D}(k,R)$ be its unbounded derived category. We denote by $\mathbf{D}_{\aone}(k,R)$ the full subcategory of $\mathbf{D}(k,R)$ consisting of $\aone$-local complexes and $\ModA(R)$ for the full subcategory of $\Mod(R)$ consisting of strictly $\aone$-invariant sheaves. We write $\Ab = \Mod(\mathbb{Z})$ and $\AbA = \ModA(\mathbb{Z})$. For $\mathcal{X} \in \Spck$, we denote by $R(\mathcal{X})$ the simplicial Nisnevich sheaf of $R$-modules freely generated by $\mathcal{X}$ and $C(\mathcal{X};R)$ for its normalized chain complex. Let $L_{\aone}$ be a left adjoint of the inclusion $\mathbf{D}_{\aone}(k,R) \hookrightarrow \mathbf{D}(k,R)$ (see \cite[Thm. 2.5]{CD09}). We write $C^{\aone}(\mathcal{X};R) = L_{\aone}(C(\mathcal{X};R))$ and $\mathbf{H}_*^{\aone}(\mathcal{X};R) = H_*(C(\mathcal{X};R))$. We consider a relative version of these definitions.

\begin{defi}
Let $f : \mathcal{X} \to \mathcal{Y}$ be a morphism of $k$-spaces. We denote by $C(f;R)$ the mapping cone of $C(\mathcal{X};R) \to C(\mathcal{Y};R)$. We write $C^{\aone}(f;R) = L_{\aone}(C(f;R))$. We define the $i$-th $\aone$-homology sheaf $\mathbf{H}^{\aone}_i(f;R)$ as the homology of $C^{\aone}(\mathcal{X};R)$ in degree $i$. When $f$ is an inclusion, we write $C^{\aone}(\mathcal{Y},\mathcal{X};R) = C^{\aone}(f;R)$ and $\mathbf{H}^{\aone}_i(\mathcal{Y},\mathcal{X};R) = \mathbf{H}^{\aone}_i(f;R)$.
\end{defi}

\subsection{Relative $\aone$-Hurewicz theorem}

We denote by $\mathcal{G}r_k^{\aone}$ the category of strongly $\aone$-invariant sheaves of groups on $\Smk$ (see \cite[Def. 1.7]{Mo2}). For a functor $F : \mathscr{C} \to \mathscr{D}$ between categories, a morphism $f : A \to A'$ in $\mathscr{D}$ is called \textit{universal with respect to $F$} if it induces a bijection $\Hom_{\mathscr{D}}(A',F(B)) \cong \Hom_{\mathscr{D}}(A,F(B))$ for every $B \in \mathscr{C}$. Morel \cite{Mo2} proved the following $\aone$-Hurewicz theorem which relates $\aone$-homotopy and $\aone$-homology.

\begin{thm}[see {\cite[Thm. 6.35 and 6.37]{Mo2}}]\label{Morel's A1 Hurewicz}
Let $(\mathcal{X},*)$ be a pointed $k$-space. Then there exists a natural morphism
\begin{equation*}
h : \pi^{\aone}_n(\mathcal{X},*) \to \mathbf{H}^{\aone}_n(\mathcal{X};\mathbb{Z})
\end{equation*}
such that if $(\mathcal{X},*)$ is $\aone$-$(n-1)$-connected for $n \geq 1$, then $h$ is universal with respect to the inclusion $\AbA \hookrightarrow \mathcal{G}r^{\aone}_k$.
\end{thm}

Morel proves that $h$ is an isomorphism assuming $k$ perfect; for the above assertion his argument works for general $k$. The following is a relative version of Theorem \ref{Morel's A1 Hurewicz}.

\begin{prop}\label{rel. A1Hure not perfect}
Let $f : (\mathcal{X},*) \to (\mathcal{Y},*)$ be a morphism from an $\aone$-simply connected pointed $k$-space to an $\aone$-connected $k$-space. Suppose that $f$ is $\aone$-$(n-1)$-connected for $n \geq 2$. Then there exists a universal morphism
\begin{equation*}
h : \pi^{\aone}_n(f) \to \mathbf{H}^{\aone}_n(f;\mathbb{Z})
\end{equation*}
with respect to the inclusion $\AbA \hookrightarrow \mathcal{G}r^{\aone}_k$.
\end{prop}

\begin{proof}
We write $\mathbf{H}_n(-) = H_n(C(-;\mathbb{Z}))$. Let $\mathrm{Ex}_{\aone}$ be the resolution functor as in \cite[\S 3.2]{MV}. By applying the relative Hurewicz theorem of simplicial sets \cite[Thm. 3.11]{GJ} to all stalks, we have a natural isomorphism $\pi^{\aone}_i(f) \cong \mathbf{H}_i(\mathrm{Ex}_{\aone}(f))$ for all $1 \leq i \leq n$ and $\mathbf{H}_0(\mathrm{Ex}_{\aone}(f)) = 0$. Thus we obtain $\mathbf{H}_i(\mathrm{Ex}_{\aone}(f)) = 0$ for every $i \leq n-1$. By the isomorphism $\pi^{\aone}_n(f) \cong \mathbf{H}_n(\mathrm{Ex}_{\aone}(f))$, there exists an isomorphism
\begin{equation}\label{rel. A1Hur. 1st eq.}
\Hom_{\Ab}(\pi^{\aone}_n(f),A) \cong \Hom_{\Ab}(\mathbf{H}_n(\mathrm{Ex}_{\aone}(f)),A)
\end{equation}
for every $A \in \AbA$. Since $\mathbf{H}_i(\mathrm{Ex}_{\aone}(f)) = 0$ for all $i \leq n-1$, the adjunction on $L_{\aone}$ shows that
\begin{align}
\begin{split}
\Hom_{\Ab}(\mathbf{H}_n(\mathrm{Ex}_{\aone}(f)),A) &\cong \Hom_{\mathbf{D}(k,\mathbb{Z})}(C(\mathrm{Ex}_{\aone}(f);\mathbb{Z}),A[n]) \\
&\cong \Hom_{\mathbf{D}(k,\mathbb{Z})}(C^{\aone}(\mathrm{Ex}_{\aone}(f);\mathbb{Z}),A[n]).
\end{split}
\end{align}
Then $H_i(C^{\aone}(\mathrm{Ex}_{\aone}(f);\mathbb{Z})) = 0$ for all $i \leq n-1$ by \cite[Thm. 6.22]{Mo2}. Thus
\begin{equation}
\Hom_{\mathbf{D}(k,\mathbb{Z})}(C^{\aone}(\mathrm{Ex}_{\aone}(f);\mathbb{Z}),A[n]) \cong \Hom_{\Ab}(\mathbf{H}^{\aone}_n(\mathrm{Ex}_{\aone}(f);\mathbb{Z}),A).
\end{equation}
The morphism of distinguished triangles
\begin{equation*}
\begin{CD}
C^{\aone}(\mathcal{X};\mathbb{Z}) @>>> C^{\aone}(\mathcal{Y};\mathbb{Z}) @>>> C^{\aone}(f;\mathbb{Z}) @>>> \\
@V{\cong}VV @V{\cong}VV @VVV\\
C^{\aone}(\mathrm{Ex}_{\aone}(\mathcal{X});\mathbb{Z}) @>>> C^{\aone}(\mathrm{Ex}_{\aone}(\mathcal{Y});\mathbb{Z}) @>>> C^{\aone}(\mathrm{Ex}_{\aone}(f);\mathbb{Z}) @>>>
\end{CD}
\end{equation*}
in $\mathbf{D}(k,\mathbb{Z})$ induced by the natural transformation $\mathrm{Id} \to \mathrm{Ex}_{\aone}$ gives a quasi-isomorphism $C^{\aone}(f;\mathbb{Z}) \to C^{\aone}(\mathrm{Ex}_{\aone}(f);\mathbb{Z})$. Therefore, we obtain 
\begin{equation}\label{rel. A1Hur. last eq.}
\Hom_{\Ab}(\mathbf{H}^{\aone}_n(\mathrm{Ex}_{\aone}(f);\mathbb{Z}),A) \cong \Hom_{\Ab}(\mathbf{H}^{\aone}_n(f;\mathbb{Z}),A).
\end{equation}
By the isomorphisms \eqref{rel. A1Hur. 1st eq.}-\eqref{rel. A1Hur. last eq.}, we have
\begin{equation}\label{rel. A1Hur. lastlast eq.}
\Hom_{\Ab}(\pi^{\aone}_n(f),A) \cong \Hom_{\Ab}(\mathbf{H}^{\aone}_n(f;\mathbb{Z}),A) = \Hom_{\AbA}(\mathbf{H}^{\aone}_n(f;\mathbb{Z}),A)
\end{equation}
for all $A \in \AbA$. On the other hand, \cite[Thm. 6.22]{Mo2} leads to the adjunction
\begin{equation}\label{adj. Mod ModA}
\xymatrix{
\Mod(R) \ar@<1ex>[rr]^-{H_0 \circ L_{\aone}} & &\ModA(R) \ar@<1ex>[ll]
},
\end{equation}
and this induces a universal morphism
\begin{equation*}
h' : \pi^{\aone}_n(f) \to H_0(L_{\aone}(\pi^{\aone}_n(f)))
\end{equation*}
with respect to $\AbA \hookrightarrow \Ab$. By the isomorphism \eqref{rel. A1Hur. lastlast eq.}, we have
\begin{equation*}
\Hom_{\AbA}(\mathbf{H}^{\aone}_n(f;\mathbb{Z}),A) \cong \Hom_{\Ab}(\pi^{\aone}_n(f),A) \cong \Hom_{\AbA}(H_0(L_{\aone}(\pi^{\aone}_n(f))),A).
\end{equation*}
Thus Yoneda's lemma in $\AbA$ shows that
\begin{equation*}
H_0(L_{\aone}(\pi^{\aone}_n(f))) \cong \mathbf{H}^{\aone}_n(f;\mathbb{Z}).
\end{equation*}
Therefore, the composite morphism
\begin{equation*}
h : \pi^{\aone}_n(f) \xrightarrow{h'} H_0(L_{\aone}(\pi^{\aone}_n(f))) \cong \mathbf{H}^{\aone}_n(f;\mathbb{Z})
\end{equation*}
is universal with respect to $\AbA \hookrightarrow \Ab$. Since $\pi^{\aone}_n(f)$ and $\mathbf{H}^{\aone}_n(f;\mathbb{Z})$ are strongly $\aone$-invariant, the morphism $h$ is universal with respect to the inclusion $\AbA \hookrightarrow \mathcal{G}r^{\aone}_k$.
\end{proof}

When $k$ is perfect, Proposition \ref{rel. A1Hure not perfect} gives an isomorphism between the relative $\aone$-homotopy and the $\aone$-homology sheaves.

\begin{cor}\label{rel. A1Hure}
Let $f$ be as in Proposition \ref{rel. A1Hure not perfect}. If $k$ is perfect, then there exists a natural isomorphism $\pi^{\aone}_n(f) \cong \mathbf{H}^{\aone}_n(f;\mathbb{Z})$.
\end{cor}

\begin{proof}
Since $\pi^{\aone}_n(f) \in \AbA$ by \cite[Cor. 6.2]{Mo2}, Yoneda's lemma in $\AbA$ gives a natural isomorphism $\pi^{\aone}_n(f) \cong \mathbf{H}^{\aone}_n(f;\mathbb{Z})$.
\end{proof}

By Corollary \ref{rel. A1Hure}, we obtain the following.

\begin{cor}\label{A1-Whitehead theorem}
Assume $k$ perfect. Let $f : (\mathcal{X},*) \to (\mathcal{Y},*)$ be a morphism of $\aone$-simply connected pointed $k$-spaces and $n \geq 2$ an integer. If $\mathbf{H}^{\aone}_i(f;\mathbb{Z}) = 0$ for all $2 \leq i \leq n$, then $f$ is $\aone$-$n$-connected.
\end{cor}

\begin{proof}
We use induction on $i$. The assertion is clear for $i = 0$. We next consider the case $i = 1$. There exists an exact sequence
\begin{equation*}
\pi^{\aone}_1(\mathcal{Y},*) \to \pi^{\aone}_1(f) \to \pi^{\aone}_0(\mathcal{X},*).
\end{equation*}
Since $\pi^{\aone}_1(\mathcal{Y},*) = \pi^{\aone}_0(\mathcal{X},*) = 0$, we have $\pi^{\aone}_1(f) = 0$. Finally, let $i \geq 2$ and $\pi^{\aone}_{i-1}(f) = 0$. Then Corollary \ref{rel. A1Hure} shows that $\pi^{\aone}_i(f) \cong \mathbf{H}^{\aone}_i(f;\mathbb{Z})  = 0$. 
\end{proof}

\subsection{$\aone$-homology and Suslin homology}

Next, we compare $\aone$-homology and Suslin homology. Let $\mathbf{NST}_k(R)$ be the category of Nisnevich sheaves with transfers with coefficients in $R$, $\mathbf{D}_{tr}(k,R)$ be the unbounded derived category of $\mathbf{NST}_k(R)$, and $R_{tr} : \Smk \to \mathbf{NST}_k(R)$ be the functor as in \cite[Def. 2.8]{MVW} (with $R$-coefficients). Following \cite{Vo00}, we denote by $\mathbf{DM}^{eff}(k,R)$ the full subcategory of $\mathbf{D}_{tr}(k,R)$ consisting of $\aone$-local complexes. Let $L^{tr}_{\aone}$ be a left adjoint of the inclusion $\mathbf{DM}^{eff}(k,R) \hookrightarrow \mathbf{D}_{tr}(k,R)$ (see \cite[Thm. 2.5]{CD09}). We write $M(X;R) = L^{tr}_{\aone}(R_{tr}(X))$ for each $X \in \Smk$. The homology sheaves $\mathbf{H}_*^S(X;R) = H_*(M(X;R))$ are called the \textit{Suslin homology sheaves} of $X$ (cf. \cite{SV}). We introduce a relative version.

\begin{defi}
Let $f : X \to Y$ be a morphism in $\Smk$. Then we denote by $R_{tr}(f)$ the mapping cone of the morphism $R_{tr}(X) \to R_{tr}(Y)$. We write $M(f;R) = L_{\aone}^{tr}(R_{tr}(f))$. We define the $i$-th Suslin homology sheaf $\mathbf{H}^S_i(f;R)$ as the homology of $M(f;R)$ in degree $i$. When $f$ is an embedding, we write $R_{tr}(Y,X) = R_{tr}(f)$ and $\mathbf{H}^S_i(Y,X;R) = \mathbf{H}^S_i(f;R)$.
\end{defi}

Let $\mathbf{NST}_k^{\aone}(R)$ be the full subcategory of $\mathbf{NST}_k(R)$ consisting of strictly $\aone$-invariant sheaves. If $f$ is a morphism in $\Smk$, we have a morphism $\mathbf{H}^{\aone}_*(f;R) \to \mathbf{H}_*^S(f;R)$ in $\ModA(R)$. The following is an analogue of the result of Asok \cite[Cor. 3.4]{As} in higher degree.

\begin{prop}\label{Comp. A1 and Sulin}
Let $f : X \to Y$ be a morphism in $\Smk$ and let $n \geq 0$. If $\mathbf{H}^{\aone}_i(f;R) = 0$ for all $i < n$, then the natural morphism $\mathbf{H}^{\aone}_n(f;R) \to \mathbf{H}^S_n(f;R)$ is universal with respect to the canonical functor $\mathbf{NST}^{\aone}_k(R) \to \ModA(R)$.
\end{prop}

\begin{proof}
By induction on $n$ and Yoneda's lemma in $\mathbf{NST}^{\aone}_k(R)$, we may assume that $\mathbf{H}^S_i(f;R) = 0$ for all $i < n$. For $A \in \mathbf{NST}_k^{\aone}(R)$, we have
\begin{align}
\Hom_{\ModA(R)}(\mathbf{H}^{\aone}_n(f;R),A) &\cong \Hom_{\mathbf{D}(k,R)}(C(f;R),A[n]), \label{eqalpha} \\
\Hom_{\mathbf{NST}^{\aone}_k(R)}(\mathbf{H}^S_i(f;R),A) &\cong \Hom_{\mathbf{D}_{tr}(k,R)}(R_{tr}(f),A[n]). \label{eqalpha2}
\end{align}
On the other hand, the adjunction $\mathbf{D}(k,R) \rightleftarrows \mathbf{D}_{tr}(k,R)$ gives
\begin{equation}
\Hom_{\mathbf{D}_{tr}(k,R)}(R_{tr}(f),A[n]) \xrightarrow{\cong} \Hom_{\mathbf{D}(k,R)}(R(f),A[n]) \label{eqbeta}
\end{equation}
such that the diagram
\begin{equation*}
\begin{CD}
\Hom_{\mathbf{NST}^{\aone}_k(R)}(\mathbf{H}^S_n(f;R),A) @>{\cong}>> \Hom_{\mathbf{D}_{tr}(k,R)}(R_{tr}(f),A[n]) \\
@VVV @V{\cong}VV \\
\Hom_{\ModA(R)}(\mathbf{H}^{\aone}_n(f;R),A) @>{\cong}>> \Hom_{\mathbf{D}(k,R)}(R(f),A[n])
\end{CD}
\end{equation*}
commutes.
\end{proof}

By Proposition \ref{Comp. A1 and Sulin} and Yoneda's lemma in $\mathbf{NST}^{\aone}_k(R)$, if $\mathbf{H}^{\aone}_i(f;R) = 0$ for all $i < n$, then $\mathbf{H}^S_i(f;R) = 0$ for all $i < n$. By using the $\aone$-Hurewicz theorem, we obtain an analogue of the Hurewicz theorem relating $\aone$-homotopy and Suslin homology.

\begin{cor}\label{Hure A1hmtpy Suslin}
(1) Let $(X,*)$ be a pointed smooth $k$-scheme which is $\aone$-$(n-1)$-connected for $n \geq 1$. Then there exists a universal morphism $\pi^{\aone}_n(X,*) \to \mathbf{H}^S_n(X;\mathbb{Z})$ with respect to the functor $\mathbf{NST}^{\aone}_k(\mathbb{Z}) \to \mathcal{G}r_k^{\aone}$.

(2) Let $f$ be an $\aone$-$(n-1)$-connected morphism of $\aone$-simply connected pointed smooth $k$-schemes for $n \geq 2$. Then there exists a universal morphism $\pi^{\aone}_n(f) \to \mathbf{H}^S_n(f;\mathbb{Z})$ with respect to the functor $\mathbf{NST}^{\aone}_k(\mathbb{Z}) \to \mathcal{G}r^{\aone}_k$.
\end{cor}

\begin{proof}
(1) follows from Proposition \ref{Comp. A1 and Sulin} and Theorem \ref{Morel's A1 Hurewicz}, and (2) follows from Propositions \ref{Comp. A1 and Sulin} and \ref{rel. A1Hure not perfect}.
\end{proof}

\section{Proof of the main theorems}\label{hom eq ex}

In this section, we prove Theorems \ref{intro dim cor A1} and \ref{intro thm ex}.

\subsection{$\aone$-Whitehead theorem with dimension bound}

For the proof of Theorems 1.1 and 1.2, we consider the Nisnevich cohomology of morphisms. For a morphism $f : X \to Y$ in $\Smk$, $A \in \Ab$ and $n \geq 0$, we define
\begin{equation*}
H_{Nis}^n(f;A) = \Hom_{\mathbf{D}(k,\mathbb{Z})}(C(f;\mathbb{Z}),A[n]).
\end{equation*}
We write $H_{Nis}^n(X,U;A) = H_{Nis}^n(i;A)$ for an embedding $i : U \hookrightarrow X$.

\begin{prop}\label{dim thm A1 and Suslin}
Let $f : X \to Y$ be a morphism in $\Smk$. We write $d = {\max\{ \dim X + 1, \dim Y\}}$.

(1) If $\mathbf{H}^{\aone}_i(f;R) = 0$ for all $i \leq d$, then $f$ induces an isomorphism $C^{\aone}(X;R) \cong C^{\aone}(Y;R)$ in $\mathbf{D}(k,R)$.

(2) If $\mathbf{H}^S_i(f;R) = 0$ for all $i \leq d$, then $f$ induces an isomorphism of motives $M(X;R) \cong M(Y;R)$ in $\mathbf{DM}^{eff}(k,R)$.
\end{prop}

\begin{proof}
(1) For each $m > d$, we only need to show that if $\mathbf{H}^{\aone}_i(f;R) = 0$ for all $i \leq m$, then $\mathbf{H}^{\aone}_{m+1}(f;R) = 0$. By \eqref{eqalpha}, there exists a natural isomorphism
\begin{equation*}
H_{Nis}^{m+1}(f;A) \cong \Hom_{\ModA(R)}(\mathbf{H}^{\aone}_{m+1}(f;R),A)
\end{equation*}
for every $A \in \ModA(R)$. The left hand side vanishes by the exact sequence
\begin{equation*}
\cdots \to H_{Nis}^m(X;A) \to H_{Nis}^{m+1}(f;A) \to H_{Nis}^{m+1}(Y;A) \to \cdots
\end{equation*}
and \cite[Thm. 1.32]{Nis}. Therefore, Yoneda's lemma in $\ModA(R)$ shows that $\mathbf{H}^{\aone}_{m+1}(f;R) = 0$.

(2) For each $m > d$, we only need to show that if $\mathbf{H}^S_i(f;R) = 0$ for all $i \leq m$, then $\mathbf{H}^{S}_{m+1}(f;R) = 0$. By \eqref{eqalpha2}, there exists a natural isomorphism
\begin{equation*}
\Hom_{\mathbf{D}_{tr}(k,R)}(R_{tr}(f),A[m+1]) \cong \Hom_{\mathbf{NST}^{\aone}_k(R)}(\mathbf{H}^S_{m+1}(f;R),A)
\end{equation*}
for every $A \in \mathbf{NST}_k^{\aone}(R)$. By \eqref{eqbeta}, the left hand side coincides with $H_{Nis}^{m+1}(f;A)$, and thus vanishes by the proof of (1). Since Yoneda's lemma in $\mathbf{NST}^{\aone}_k(R)$ shows that $\mathbf{H}^{\aone}_{m+1}(f;R) = 0$, we have $M(f;R) = 0$.
\end{proof}

We can now prove the $\aone$-Whitehead theorem with dimension bound.

\begin{thm}\label{dim cor A1}
Assume $k$ perfect. Let $f : (X,*) \to (Y,*)$ be a morphism of $\aone$-simply connected pointed smooth $k$-schemes. If $\mathbf{H}^{\aone}_i(f;R) = 0$ for all $2 \leq i \leq {\max\{ \dim X + 1, \dim Y\}}$, then $f$ is an $\aone$-weak equivalence.
\end{thm}

\begin{proof}
The morphism $\mathbf{H}^{\aone}_0(X;\mathbb{Z}) \to \mathbf{H}^{\aone}_0(Y;\mathbb{Z})$ is an isomorphism by \cite[Prop. 3.5]{As}. Moreover, $\mathbf{H}^{\aone}_1(X;\mathbb{Z}) = \mathbf{H}^{\aone}_1(Y;\mathbb{Z}) = 0$ by \cite[Thm. 6.35]{Mo2}. Thus we have $\mathbf{H}^{\aone}_0(f;\mathbb{Z}) = \mathbf{H}^{\aone}_1(f;\mathbb{Z}) = 0$. Therefore, our assumption and Proposition \ref{dim thm A1 and Suslin} show that $\mathbf{H}^{\aone}_i(f;\mathbb{Z}) = 0$ for all $i \in \mathbb{Z}$. Thus we have $\pi^{\aone}_i(f) = 0$ for every $i \geq 0$ by Corollary \ref{A1-Whitehead theorem}. Since $f$ induces $\pi^{\aone}_i(X,*) \cong \pi^{\aone}_i(Y,*)$ for all $i \geq 0$, the morphism $f$ is an $\aone$-weak equivalence by \cite[Prop. 2.14]{MV}.
\end{proof}

Theorem \ref{dim cor A1} implies the following.

\begin{cor}\label{cor dim cor A1}
Assume $k$ perfect. Let $f : (X,*) \to (Y,*)$ be a morphism of pointed smooth $k$-schemes and let $d = {\max\{ \dim X + 1, \dim Y\}}$. Suppose that $(X,*)$ is $\aone$-simply connected and $(Y,*)$ is $\aone$-connected. If $f$ is $\aone$-$d$-connected, then $f$ is an $\aone$-weak equivalence.
\end{cor}

\begin{proof}
By $d \geq 1$, the exact sequence
\begin{equation*}
0 = \pi^{\aone}_1(X,*) \to \pi^{\aone}_1(Y,*) \to \pi^{\aone}_1(f) = 0
\end{equation*}
shows that $(Y,*)$ is $\aone$-simply connected. On the other hand, by Proposition \ref{rel. A1Hure not perfect}, $\mathbf{H}^{\aone}_i(f;R) = 0$ for all $i \leq d$. Thus $f$ is an $\aone$-weak equivalence by Theorem \ref{dim cor A1}.
\end{proof}

Proposition \ref{dim thm A1 and Suslin} also has the following application.

\begin{cor}\label{susp a1}
Assume $k$ perfect. Let $f : (X,*) \to (Y,*)$ be a morphism of pointed smooth $k$-schemes. We assume that $\mathbf{H}^{\aone}_i(f;R) = 0$ for all $i \leq {\max\{ \dim X + 1, \dim Y\}}$. Then the morphism $S^2 \wedge f : S^2 \wedge X \to S^2 \wedge Y$ is an $\aone$-weak equivalence. Moreover, if $X$ and $Y$ are $\aone$-connected, then the morphism $S^1 \wedge f : S^1 \wedge X \to S^1 \wedge Y$ is an $\aone$-weak equivalence.
\end{cor}

\begin{proof}
For a $k$-space $\mathcal{X}$, the suspension $S^1 \wedge \mathcal{X}$ is $\aone$-connected by \cite[Thm. 6.38]{Mo2}. Similarly, if $\mathcal{X}$ is $\aone$-connected, then $S^1 \wedge \mathcal{X}$ is $\aone$-simply connected. Since $f$ induces isomorphisms for all $\aone$-homology sheaves by Proposition \ref{dim thm A1 and Suslin}, so does $S^1 \wedge f$. Therefore, Corollary \ref{A1-Whitehead theorem} shows that $S^2 \wedge f$ is an $\aone$-weak equivalence. Similarly, $S^1 \wedge f$ is an $\aone$-weak equivalence when $X$ and $Y$ are $\aone$-connected.
\end{proof}

\subsection{$\aone$-excision theorem}

We next prove an excision theorem for $\aone$- and Suslin homology. This is an analogue of \cite[Prop. 3.8]{As} in higher degree.

\begin{thm}\label{thm. A1-excision}
Let $X$ be a smooth $k$-scheme and $U$ a Zariski open set of $X$ whose complement has codimension $r$. Then the morphisms $\mathbf{H}^{\aone}_i(U;R) \to \mathbf{H}^{\aone}_i(X;R)$ and $\mathbf{H}^S_i(U;R) \to \mathbf{H}^S_i(X;R)$ are isomorphisms for every $i < r - 1$ and epimorphisms for $i = r - 1$.
\end{thm}

\begin{proof}
It suffices to prove that $\mathbf{H}^{\aone}_i(X,U;R) = \mathbf{H}^S_i(X,U;R) = 0$ for all $i < r$. By Proposition \ref{Comp. A1 and Sulin}, we only need to prove this for the $\aone$-homology sheaves. We use induction on $i$. The case $i < 0$ follows from \cite[Thm. 6.22]{Mo2}. We assume $i \geq 0$ and $\mathbf{H}^{\aone}_j(X,U;R) = 0$ for all $j < i$. By \eqref{eqalpha}, we have
\begin{equation*}
H_{Nis}^i(X,U;A) \cong \Hom_{\ModA(R)}(\mathbf{H}^{\aone}_i(X,U;R),A)
\end{equation*}
for every $A \in \ModA(R)$. Then the left hand side vanishes by \cite[Lem. 6.4.4]{Mo1}. Therefore, we have $\mathbf{H}^{\aone}_i(X,U;R) = 0$ by Yoneda's lemma in $\ModA(R)$.
\end{proof}

Theorem \ref{thm. A1-excision} gives the excision theorem for $\aone$-homotopy.

\begin{cor}\label{A1 excision of homotopy}
Assume $k$ perfect. Let $(X,*)$ be a pointed smooth $k$-scheme and $(U,*)$ a pointed Zariski open set of $(X,*)$ whose complement has codimension $r$. If $(X,*)$ and $(U,*)$ are $\aone$-simply connected, then the morphism $\pi^{\aone}_i(U) \to \pi^{\aone}_i(X)$ is an isomorphism for every $i < r - 1$ and an epimorphism for $i = r - 1$. In other words, the pair $(X,U)$ is $\aone$-$(r-1)$-connected.
\end{cor}

\begin{proof}
Since $\mathbf{H}^{\aone}_i(X,U;\mathbb{Z}) = 0$ for all $i < r$ by Theorem \ref{thm. A1-excision}, the pair $(X,U)$ is $\aone$-$(r-1)$-connected by Corollary \ref{A1-Whitehead theorem}.
\end{proof}

\section{$\aone$-homology of a hyperplane embedding}\label{sec. P1}

In this section, as an example of relative $\aone$-homology, we compute the $\aone$-homology of a hyperplane embedding $\mathbb{P}^{n-1} \hookrightarrow \mathbb{P}^{n}$ in degree $\leq n$. By a suitable linear change of coordinates, we may regard $\mathbb{P}^{n-1}$ as the hyperplane in $\mathbb{P}^n$ defined by $x_n = 0$, where $x_0, \ldots, x_n$ denote homogeneous coordinates on $\mathbb{P}^n$. Let $\underline{\mathbf{K}}^{MW}_n$ be the unramified Milnor-Witt $K$-theory defined by Morel \cite{Mo2}. For $A \in \Ab$, we write $A \otimes^{\aone} R = H_0(L_{\aone}(A \otimes R))$ which is called \textit{$\aone$-tensor product} by Morel \cite{Mo1}. Our main result is the following.

\begin{thm}\label{(Pn,Pn-1)}
For $0 \leq i \leq n$, $n > 0$, we have
\begin{equation*}
\mathbf{H}^{\aone}_i(\mathbb{P}^n,\mathbb{P}^{n-1};R) \cong 
\begin{cases}
\underline{\mathbf{K}}^{MW}_n \otimes^{\aone} R &(i = n) \\
0 &(i < n).
\end{cases}
\end{equation*}
In particular, when $i < n$, we have $\mathbf{H}^{\aone}_i(\mathbb{P}^n;R) \cong \mathbf{H}^{\aone}_i(\mathbb{P}^{i+1};R)$.
\end{thm}

When $i = n$, we have the following description.

\begin{cor}\label{cor (Pn,Pn-1)}
There exists a morphism $\underline{\mathbf{K}}^{MW}_{n+1} \otimes^{\aone} R \to \mathbf{H}^{\aone}_{n}(\mathbb{P}^{n};R)$ such that
\begin{equation*}
\mathbf{H}^{\aone}_n(\mathbb{P}^{n+1};R) \cong \Coker (\underline{\mathbf{K}}^{MW}_{n+1} \otimes^{\aone} R \to \mathbf{H}^{\aone}_{n}(\mathbb{P}^{n};R)).
\end{equation*}
\end{cor}

\begin{proof}
By Theorem \ref{(Pn,Pn-1)}, the homology exact sequence
\begin{equation*}
\mathbf{H}^{\aone}_{n+1}(\mathbb{P}^{n+1},\mathbb{P}^n;R) \to \mathbf{H}^{\aone}_{n}(\mathbb{P}^n;R) \to \mathbf{H}^{\aone}_{n}(\mathbb{P}^{n+1};R) \to \mathbf{H}^{\aone}_{n}(\mathbb{P}^{n+1},\mathbb{P}^n;R) = 0
\end{equation*}
gives an isomorphism
\begin{align*}
\mathbf{H}^{\aone}_{n}(\mathbb{P}^{n+1};R) &\cong \Coker(\mathbf{H}^{\aone}_{n+1}(\mathbb{P}^{n+1},\mathbb{P}^n;R) \to \mathbf{H}^{\aone}_{n}(\mathbb{P}^n;R)) \\
&\cong \Coker (\underline{\mathbf{K}}^{MW}_{n+1} \otimes^{\aone} R \to \mathbf{H}^{\aone}_{n}(\mathbb{P}^{n};R)).
\end{align*}
\end{proof}

Theorem \ref{(Pn,Pn-1)} and Proposition \ref{Comp. A1 and Sulin} imply the following.

\begin{cor}\label{cor (Pn,Pn-1) sus}
For all $i < n$, we have $\mathbf{H}^{S}_i(\mathbb{P}^n,\mathbb{P}^{n-1};R) = 0$. Moreover, there exists a universal morphism $\underline{\mathbf{K}}^{MW}_n \otimes^{\aone} R \to \mathbf{H}^{S}_n(\mathbb{P}^n,\mathbb{P}^{n-1};R)$ with respect to the canonical functor $\mathbf{NST}^{\aone}_k(R) \to \ModA(R)$. In particular, when $i < n$ we have
\begin{equation*}
\mathbf{H}^{S}_i(\mathbb{P}^n;R) \cong \mathbf{H}^{S}_i(\mathbb{P}^{i+1};R).
\end{equation*}
\end{cor}

\begin{rem}
The vanishing $\mathbf{H}^{\aone}_i(\mathbb{P}^n,\mathbb{P}^{n-1};R) = 0$ for $i < n$ is an example where the Lefschetz hyperplane theorem holds. However, it is not true in general that $\mathbf{H}^{\aone}_i(X,H;R) = 0$ for $i < \dim X$ and $H \subseteq X$ a very ample divisor. Indeed, let $C \subseteq \mathbb{P}^2$ be a smooth plane curve of degree $\geq 3$. Since $C$ is not rational, it is not $\aone$-connected by \cite[Prop. 2.1.12]{AM}. Therefore, the canonical morphism $\mathbf{H}^{\aone}_0(C;R) \to R$ is not an isomorphism by \cite[Thm. 4.14]{As}. Thus the morphism $\mathbf{H}^{\aone}_0(C;R) \to \mathbf{H}^{\aone}_0(\mathbb{P}^2;R) \cong R$ is not an isomorphism.
\end{rem}

\subsection{A basic distinguished triangle}

For the proof of Theorem \ref{(Pn,Pn-1)}, we compute the mapping cone of $C^{\aone}(\mathbb{A}^{n}-\{0\};R) \to C^{\aone}(\mathbb{P}^{n-1};R)$. We first give a Zariski excision result for $\aone$-homology.

\begin{lem}\label{A1 homology excision thm}
Let $\{U, V\}$ be a Zariski covering of a smooth $k$-scheme $X$. Then the morphism $(U,U \cap V) \to (X,V)$ induces a quasi-isomorphism
\begin{equation*}
C^{\aone}(U,U \cap V;R) \cong C^{\aone}(X,V;R).
\end{equation*}
\end{lem}

\begin{proof}
Since the functor $L_{\aone}$ is exact (see, \textit{e.g.}, \cite[Thm. 2.5]{CD09}), we only need to show that the morphism ${(U,U \cap V) \to (X,V)}$ induces a quasi-isomorphism
\begin{equation}\label{eq zariski ex}
C(U,U \cap V;R) \cong C(X,V;R).
\end{equation}
For each open set $W \subseteq X$, we regard $R(W)$ as a subsheaf of $R(X)$. Then by the short exact sequence in \cite[proof of Prop.3.32]{AD}, we see that $R(U \cap V) = R(U) \cap R(V)$ and $R(X) = R(U) + R(V)$. Thus we have an isomorphism
\begin{equation*}
R(U)/R(U \cap V) = R(U)/(R(U) \cap R(V)) \cong (R(U) + R(V))/R(V) = R(X)/R(V).
\end{equation*}
Finally, $R(U)/R(U \cap V)$ and $R(X)/R(V)$ are canonically quasi-isomorphic to $C(U,U \cap V;R)$ and $C(X,V;R)$, respectively. Therefore, we obtain \eqref{eq zariski ex}.
\end{proof}

We obtain the following distinguished triangle in $\mathbf{D}(k,R)$.

\begin{prop}\label{dist tr of A1 of proj sp}
For $n \geq 1$ and $* \in \mathbb{P}^n(k)$, we have a distinguished triangle
\begin{equation*}
C^{\aone}(\mathbb{A}^{n}-\{0\};R) \to C^{\aone}(\mathbb{P}^{n-1};R) \to C^{\aone}(\mathbb{P}^{n},*;R) \to
\end{equation*}
in $\mathbf{D}(k,R)$, where the first morphism is induced by the $\mathbb{G}_m$-quotient and the second morphism is induced by $(\mathbb{P}^{n-1},\emptyset) \to (\mathbb{P}^{n},*)$.
\end{prop}

\begin{proof}
We write $U = {\mathbb{P}^{n} - \{ (0:\ldots:0:1) \}}$. Since the projection $\rho : U \to \mathbb{P}^{n-1}$ is a vector bundle, it is an $\aone$-weak equivalence by \cite[Example 2.2]{MV}. We denote by $V$ the Zariski open set of $\mathbb{P}^{n}$ defined by $x_n \neq 0$. Since the diagram
\begin{equation*}
\begin{CD}
\mathbb{A}^{n}-\{0\} @>{/\mathbb{G}_m}>> \mathbb{P}^{n-1} \\
@A{\cong}AA @A{\rho}AA \\
U \cap V @>{\subseteq}>> U
\end{CD}
\end{equation*}
commutes, we obtain the commutative diagram
\begin{equation}\label{diagram dist proj 1}
\begin{CD}
C^{\aone}(\mathbb{A}^{n}-\{0\};R) @>>> C^{\aone}(\mathbb{P}^{n-1};R) \\
@V{\cong}VV @V{\cong}VV \\
C^{\aone}(U \cap V;R) @>>> C^{\aone}(U;R).
\end{CD}
\end{equation}
On the other hand, Lemma \ref{A1 homology excision thm} gives the commutative diagram
\begin{equation}\label{diagram dist proj 2}
\begin{CD}
C^{\aone}(\mathbb{P}^{n-1};R) @>>> C^{\aone}(\mathbb{P}^{n},V;R) \\
@V{\cong}VV @V{\cong}VV \\
C^{\aone}(U;R) @>>> C^{\aone}(U,U \cap V;R).
\end{CD}
\end{equation}
By the diagrams \eqref{diagram dist proj 1} and \eqref{diagram dist proj 2}, we obtain an isomorphism of triangles
\begin{equation*}
\begin{CD}
C^{\aone}(\mathbb{A}^{n}-\{0\};R) @>>> C^{\aone}(\mathbb{P}^{n-1};R) @>>> C^{\aone}(\mathbb{P}^{n-1},V;R) @>>> \\
@V{\cong}VV @V{\cong}VV @V{\cong}VV \\
C^{\aone}(U \cap V;R) @>>> C^{\aone}(U;R) @>>> C^{\aone}(U,U \cap V;R) @>>>.
\end{CD}
\end{equation*}
Since the lower triangle is distinguished, so is the upper triangle. Therefore, we only need to show that for a $k$-rational point $* \in V(k)$, the morphism $(\mathbb{P}^{n-1},*) \to (\mathbb{P}^{n-1},V)$ induces a quasi-isomorphism $C^{\aone}(\mathbb{P}^{n},*;R) \cong C^{\aone}(\mathbb{P}^{n-1},V;R)$. Since $V \cong \mathbb{A}^{n}$, the exact sequence
\begin{equation*}
0 = \mathbf{H}^{\aone}_i(V;R) \to \mathbf{H}^{\aone}_i(\mathbb{P}^{n};R) \to \mathbf{H}^{\aone}_i(\mathbb{P}^{n},V;R) \to \mathbf{H}^{\aone}_{i-1}(V;R) = 0
\end{equation*}
shows that $\mathbf{H}^{\aone}_i(\mathbb{P}^{n};R) \cong \mathbf{H}^{\aone}_i(\mathbb{P}^{n},V;R)$ for all $i \geq 2$. For $i = 0,1$, there exists an exact sequence
\begin{multline*}
0 = \mathbf{H}^{\aone}_1(V;R) \to \mathbf{H}^{\aone}_1(\mathbb{P}^{n};R) \to \mathbf{H}^{\aone}_1(\mathbb{P}^{n},V;R) \\
\to \mathbf{H}^{\aone}_{0}(V;R) \to \mathbf{H}^{\aone}_0(\mathbb{P}^{n};R) \to \mathbf{H}^{\aone}_0(\mathbb{P}^{n},V;R) \to 0.
\end{multline*}
Then the morphism $\mathbf{H}^{\aone}_{0}(V;R) \to \mathbf{H}^{\aone}_0(\mathbb{P}^{n};R)$ is an isomorphism by \cite[Prop. 3.5]{As}. Therefore, we have $\mathbf{H}^{\aone}_0(\mathbb{P}^n,*;R) = \mathbf{H}^{\aone}_0(\mathbb{P}^{n},V;R) = 0$ and $\mathbf{H}^{\aone}_1(\mathbb{P}^{n};R) \cong \mathbf{H}^{\aone}_1(\mathbb{P}^{n},V;R)$. Thus $C^{\aone}(\mathbb{P}^{n},*;R) \to C^{\aone}(\mathbb{P}^{n-1},V;R)$ is a quasi-isomorphism.
\end{proof}

\subsection{Proof of Theorem \ref{(Pn,Pn-1)}}

We can now prove Theorem \ref{(Pn,Pn-1)}.

\begin{proof}[(Proof of Theorem \ref{(Pn,Pn-1)})]
We first prove $\mathbf{H}^{\aone}_i(\mathbb{P}^n,\mathbb{P}^{n-1};R) = 0$ for all $i < n$. The $\aone$-weak equivalence $\rho$ as in the proof of Proposition \ref{dist tr of A1 of proj sp} gives an isomorphism $\mathbf{H}^{\aone}_i(\mathbb{P}^{n-1};R) \xrightarrow{\cong} \mathbf{H}^{\aone}_i(\mathbb{P}^n - \{0\};R)$. On the other hand, $\mathbf{H}^{\aone}_i(\mathbb{P}^n - \{0\};R) \to \mathbf{H}^{\aone}_i(\mathbb{P}^n;R)$ is an isomorphism for all $i < n-1$ and an epimorphism for $i = n-1$ by Theorem \ref{thm. A1-excision}. Thus we have $\mathbf{H}^{\aone}_i(\mathbb{P}^n,\mathbb{P}^{n-1};R) = 0$.

Next, we prove $\mathbf{H}^{\aone}_n(\mathbb{P}^n,\mathbb{P}^{n-1};R) \cong \underline{\mathbf{K}}^{MW}_n \otimes^{\aone} R$ for all $n \geq 2$. By Proposition \ref{dist tr of A1 of proj sp}, there exists a morphism of distinguished triangles
\begin{align*}
\begin{CD}
C^{\aone}(\mathbb{P}^{n-1};R) @>>> C^{\aone}(\mathbb{P}^n;R) @>>> C^{\aone}(\mathbb{P}^n,\mathbb{P}^{n-1};R) @>>> \\
@| @V{\alpha}VV @V{\beta}VV \\
C^{\aone}(\mathbb{P}^{n-1};R) @>>> C^{\aone}(\mathbb{P}^n,*;R) @>>> C^{\aone}(\mathbb{A}^n - \{0\};R)[1] @>>>,
\end{CD}
\end{align*}
where $\alpha$ is induced by $(\mathbb{P}^n,\emptyset) \to (\mathbb{P}^n,*)$. Taking the homology exact sequence, we obtain $\mathbf{H}^{\aone}_n(\mathbb{P}^n,\mathbb{P}^{n-1};R) \cong \mathbf{H}^{\aone}_{n-1}(\mathbb{A}^n - \{0\};R)$ for all $n \geq 2$ by the five lemma. Note that the adjunctions $\Ab \rightleftarrows \Mod(R)$ and \eqref{adj. Mod ModA} show that the functor $-\otimes^{\aone}R : \Ab \to \ModA(R)$ is left adjoint to the canonical functor $\ModA(R) \to \Ab$. Moreover, for every $X \in \Smk$ which is $\aone$-$(n-1)$-connected, the adjunction $\Ab \rightleftarrows \ModA(R)$ leads to a natural isomorphism
\begin{equation}\label{aone tensor eq}
\mathbf{H}_{n}^{\aone}(X;R) \cong \mathbf{H}_{n}^{\aone}(X;\mathbb{Z}) \otimes^{\aone} R.
\end{equation}
Hence, we have
\begin{align}
\Hom_{\Mod(R)}(\mathbf{H}^{\aone}_{n-1}(\mathbb{A}^n - \{0\};R),A) 
&\cong \Hom_{\Ab}(\mathbf{H}^{\aone}_{n-1}(\mathbb{A}^n - \{0\};\mathbb{Z}),A) \label{last eq 2} \\
&\cong \Hom_{\Ab}(\pi^{\aone}_{n-1}(\mathbb{A}^n - \{0\}),A) \label{last eq 3} \\
&\cong \Hom_{\Ab}(\underline{\mathbf{K}}^{MW}_n,A) \label{last eq 4} \\
&\cong \Hom_{\ModA(R)}(\underline{\mathbf{K}}^{MW}_n \otimes^{\aone} R,A), \label{last eq 5}
\end{align}
where \eqref{last eq 2} and \eqref{last eq 5} follow from the adjunction ${\Ab \rightleftarrows \ModA}$, \eqref{last eq 3} from Theorem \ref{Morel's A1 Hurewicz} and \eqref{last eq 4} from \cite[Thm. 6.40]{Mo2}. Thus Yoneda's lemma in $\ModA(R)$ shows that
\begin{equation*}
\mathbf{H}^{\aone}_n(\mathbb{P}^n,\mathbb{P}^{n-1};R) \cong \mathbf{H}^{\aone}_{n-1}(\mathbb{A}^n - \{0\};R) \cong \underline{\mathbf{K}}^{MW}_n \otimes^{\aone} R.
\end{equation*}

Finally, we prove $\mathbf{H}^{\aone}_1(\mathbb{P}^1;R) \cong \underline{\mathbf{K}}^{MW}_1 \otimes^{\aone} R$. For $R  = \mathbb{Z}$, this is a direct consequence of \cite[Lem. 2.15 and Cor. 2.18]{MV} and \cite[Thm. 3.37]{Mo2}. Since $\mathbb{P}^1$ is $\aone$-connected, the general case follows from \eqref{aone tensor eq}.
\end{proof}


\end{document}